\documentclass[12pt]{amsart}
\usepackage{amsmath,amsfonts,latexsym, ulem}
\usepackage{amscd,amssymb,amsmath}
\usepackage{amsfonts}
\usepackage{amsbsy}
\usepackage{subfigure,tikz}
\usepackage{hyperref}

\usepackage[square,numbers]{natbib}
\usepackage{mathtools}
\usepackage{microtype}

\newtheorem{lema}{Lemma}[section]

\newtheorem{remark}[lema]{Remark}
\newtheorem{definition}[lema]{Definition}
\newtheorem{example}[lema]{Example}

\newcommand{\bea}{\begin{eqnarray*}}
\newcommand{\eea}{\end{eqnarray*}}
\newcommand{\zz}[1]{}

\newtheorem{theorem}[lema]{Theorem}
\newtheorem{proposition}[lema]{Proposition}
\newtheorem{corollary}[lema]{Corollary}
\newtheorem*{mainproblem}{Main Problem}

\newcommand{\g}{\mathrm g}
\newcommand{\tr}{\mathrm tr}

\newcommand{\Mod}{\mathrm{Mod}}
 \newcommand{\NN}{{\mathbb{N}}}
 \newcommand{\ZZ}{{\mathbb{Z}}}
 \newcommand{\RR}{{\mathbb{R}}}
 \newcommand{\CC}{{\mathbb{C}}}
 \newcommand{\QQ}{{\mathbb{Q}}}

\newcommand{\reg}{{\rm reg}}

\newcommand{\per}{{\rm Per}}

\newcommand{\cA}{{\mathcal A}}

\newcommand{\cAP}{{\mathcal AP}}

\newcommand{\cU}{{\mathcal U}}
\begin{document}

\title[Realizability of the algebraic periods for  Morse--Smale diffeomorphisms]{Every finite set of natural numbers is realizable as algebraic periods  of a Morse--Smale diffeomorphism}
\author[G. Graff]{Grzegorz Graff$^*$}
\thanks{Research supported by the National Science Centre, Poland,
within the grant Sheng 1 UMO-2018/30/Q/ST1/00228. The third author was also supported by the Slovenian Research and Innovation Agency program P1-0292 and grant J1-4001.}
\address{$^*$ \; Faculty of Applied Physics and Mathematics, Gda\'nsk University of
Technology, Narutowicza 11/12, 80-233 Gda\'nsk, Poland}
\email{grzegorz.graff@pg.edu.pl}
\author[W. Marzantowicz]{Wac{\l}aw Marzantowicz$^{**}$}
\address{$^{**}$ \;Faculty of Mathematics and Computer Science, Adam Mickiewicz University, Pozna{\'n}, ul. Uniwersytetu Pozna{\'n}skiego 4, 61-614 Pozna{\'n}, Poland}
\email{marzan@amu.edu.pl}

\author[{\L}. P. Michalak]{{\L}ukasz Patryk Michalak$^{**}$\,'\,$^{***}$}
\address{$^{***}$ \;Institute of Mathematics, Physics and Mechanics, Jadranska 19, SI-1000 Ljubljana, Slovenia}
\email{lukasz.michalak@amu.edu.pl}
\author[A. Myszkowski]{Adrian Myszkowski$^{\diamond}$}
 \address{$^\diamond$ \;Independent Researcher}
\email{adrian.myszkowski@pg.edu.pl}


\begin{abstract}

A given self-map $f\colon M\to M$ of a compact manifold determines the sequence $(L(f^n))$ of the Lefschetz numbers of its iterations. We consider its dual sequence $(a_n(f))$ given by the M{\"o}bius inversion formula.  The set $\cAP(f)=\{ n\in \mathbb N \colon a_n(f)\neq 0\}$ is called the set of algebraic periods.

We solve an open problem existing in literature by showing that for every finite subset $\cA$ of natural numbers, there exist an orientable surface $S_{\rm g}$, as well as a non-orientable surface $N_\g$,  of genus ${\rm g}$,  and a Morse--Smale diffeomorphism  $f$ of this surface such that $\mathcal{AP}(f)=\mathcal{A}$.
For such a map it implies the existence of points of a minimal period $n$ for each odd $n \in \mathcal{A}$.
For the orientation-reversing Morse--Smale diffeomorphisms of $S_\g$, we identify strong  restrictions on $\cAP(f)$.
Our method  also provides an estimate of the number of conjugacy classes of mapping classes containing Morse--Smale diffeomorphisms, which is exponential in  $\g$.

\end{abstract}
\maketitle

\section{Introduction} \label{sec:Introduction}

Let $f \colon X\to X$ be a self-map of a topological space $X$. One of the classical problems of  the topological theory of dynamical systems is to find $n$-periodic points of $f$, i.e. $x\in X$ such that $f^n(x)=x$ for some $n\in \NN$.
If $ f^k(x)\neq x$ for $k<n$, then $n$ is called the minimal period of $f$ at $x$, and $x$ is a point of period $n$. Then, the set $\{x, f(x), f^2(x), \, \dots\,, f^{n-1}(x)\}$ is called the orbit of $x$.
We
denote by $P^n(f)={\rm Fix}(f^n)$ the set of all $n$-periodic points, and by $P_n(f)$ the set of  all points  of period $n$. We have $P_n(f) = P^n(f) \setminus
{\underset{k|n, k<n}{\,\bigcup\,}}P^k(f)$.

An important invariant describing the dynamics of $f$ is the set of all minimal periods of $f$,  denoted by $\per (f)=\{n\in \NN:\, P_n(f) \neq \emptyset\}$ .

One of the natural problems in the field of periodic point theory is the following question: {\it Supposing that  $\cA$  is a subset of natural numbers, is it possible to realize it as the minimal periods of a self-map of some topological space $X$?} This problem has been extensively studied (mainly from the arithmetical point of view) by many authors (cf. \cite{EPPW} and the expository paper \cite{BGW}), even in the stronger version when the number of orbits of  each minimal period is also given.  In general, it turns out that there are no obstacles for such realizability for a very general class of maps and relatively simple spaces.
In particular, any (also infinite) sequence of natural numbers can be realized for smooth maps on a two-dimensional torus $\mathbb T^2$ \cite{Windsor}.
However, the approach in \cite{Windsor} provides the realization in the homotopy class of the identity map, so the periodic points are not detectable by topological/homotopical methods. In this paper, we address a more subtle version of the problem from the perspective of homotopy invariants and
for a more restricted class of maps (see Main Problem below). We also discuss the relations of our problem with the existence of periodic points (refer to Proposition~\ref{prop:periodic for transversal} and Corollary \ref{coro: main analytical}).

\begin{mainproblem}
	Can any finite set $\cA$ of  natural numbers be the set of algebraic  periods (see Definition \ref{defi:algebraic periods}) of a $C^1$ Morse–Smale diffeomorphism
	on a closed surface (orientable or non-orientable) of some genus $g$?
\end{mainproblem}
Recall that every Morse--Smale diffeomorphism has only a finite number of periodic points, which follows easily from its definition (cf. \cite{Smale}).

Below, we sketch the motivation to study  the problem in such a formulation.  The tools for posing and solving this problem are closely related to the classical concept of algebraic topology, namely the Lefschetz number $L(f)$, which is homotopy invariant (see \cite{Jez-Mar} for a comprehensive overview and details).

With the sequence $(L(f^n))$,  we associate the dual sequence $(a_n(f))$ given by the M{\"o}bius inversion formula
\begin{equation}\label{equa:dual sequences}
	a_n(f)=\frac{1}{n}\, \sum_{k\mid n}\,  \mu(n/k) \, L(f^k) = \frac{1}{n}\, \sum_{k\mid n}\, \mu(k) L(f^{n/k}),
\end{equation}
where $\mu\colon \NN \to \{-1, 0, 1\}$  is the M{\"o}bius function.
From the M{\"o}bius  inversion formula, it follows that $\;\;
L(f^n)=  \sum_{k\mid n}\,  k\, a_k(f).$
By definition, $a_n(f) \in \QQ$,  but in fact  they are integers due to  Dold's theorem (cf. \cite{Dold}).

\begin{definition}\label{defi:algebraic periods}
	The set
	$\;\, \{n\in \NN:\, a_n(f) \,\neq \,  0\}\;\,$
	is called {\it the set of algebraic periods} of $f$ and is denoted $\cAP (f)$.
\end{definition}

\begin{example}\label{rema:algebraic periods for finite sets}\rm
	Let $X$ be a finite set and $f\colon X\to X$ a map. Then, $L(f^n) = \vert P^n(f) \vert $, and consequently $a_n(f) = \vert P_n(f)\vert$. Thus, in this example,  $a_n(f) \neq 0 $ implies $ P_n(f) \neq \emptyset $ which gives $n\in \per (f)$.
\end{example}

In general, the non-vanishing of $a_n(f)$ does not imply the existence of points of period $n$. However, for some special classes of smooth self-maps of manifolds, such  as the Morse--Smale diffeomorphisms or transversal $C^1$ maps,
$a_n(f) \neq 0$ implies $P_n(f) \neq \emptyset$ if $n$ is odd, or $P_n(f) \cup P_{\frac{n}{2}}(f)  \neq \emptyset$ if $n$ is even (see Proposition \ref{prop:periodic for transversal}).
The above led Llibre, Sirvent, and other collaborators to  study $\cAP(f)$, or more precisely its odd part $\cAP_{odd}(f)=\cAP(f) \cap (2\NN-1)$. They had been considering   situations when it is computable by direct algebraic topology methods (see \cite{LlSi3} for an exposition, and the bibliography  of the results of this group up to 2013).
In fact, instead of algebraic periods, they considered a subset of natural numbers called ``the minimal set of Lefschetz periods" denoted by ${\rm MPer}_L(f)$ (see Definition \ref{defi::minimal Lefschetz periods}). Later, these authors noted  that this set does not contain even numbers (cf. \cite{LlSi2}). Recently, Graff, Lebied\'z, and Myszkowski showed in \cite{GLM}  that ${ \rm MPer}_L(f) = \cAP_{odd}(f) $, which justifies to use  only the notion of algebraic periods. The proof of \cite{GLM}  uses the periodic expansion (cf.~the~formula~(\ref{formula:k-periodic expansion})) and  is based on the local expression of the fixed point index as a combination of basic periodic sequences. In the appendix,  we provide a purely algebraic proof of the fact that ${ \rm MPer}_L(f) = \cAP_{odd}(f) $ in Theorem~\ref{prop:final algebraic periods}.

A direct inspiration for this work was the papers  \cite{LlSi1, LlSi3} and the following question posed there:

{\it
	Can any finite set $\cA$ of odd positive integers be the minimal set of Lefschetz periods for a $C^1$ Morse–Smale diffeomorphism on some orientable/non-orientable
	compact surface without boundary with a convenient genus $\g$?}

As stated by the authors in \cite{LlSi1}, in all their works they do not characterize the sets of homotopy classes of Morse--Smale diffeomorphisms.
They only described the sets called the minimal sets of Lefschetz periods ${\rm MPer}_L(f)$ (Definition \ref{defi::minimal Lefschetz periods}), i.e. $\cAP(f) \setminus 2\NN$ of  quasi-unipotent homeomorphisms
(i.e. homeomorphisms for which the induced linear map $H_*(f)$  on the homology of $M$ has spectrum consisting of the roots of unity only).
Remind that if $f \colon M\to M$ is a Morse--Smale diffeomorphism, then $f$ is quasi-unipotent by the Shub result \cite{Shub}, but not conversely. Consequently, their results gave only the necessary conditions, i.e. the restrictions on the algebraic periods of Morse--Smale diffeomorphisms never providing a geometrical realization.   These restrictions were obtained by purely algebraic computation of the Lefschetz zeta function based on known forms of the homology groups, or cohomology rings, of studied  manifolds.

In this work, by a direct geometrical argument we give a complete positive answer to the Main Problem, i.e. a   stronger version of the stated above question Llibre and Sirvent, proving that for any finite set  $\cA \subset \NN$ there exists an orientable surface of genus $\g$ and a preserving orientation diffeomorphism $f \colon S_\g\to S_\g$  such that $\cAP(f)= \cA$ (Theorem \ref{thm:main1}).
In the case in which demand that the realization is performed in the class of reversing orientation Morse--Smale diffeomorphism, we find strong restrictions on $\cAP(f)$, namely
$\cAP(f) \cap (2\NN -1)= \emptyset$, i.e. the set of algebraic periods consists only of even numbers if $f$ is a reversing orientation diffeomorphism of an orientable surface, which was shown by A.~Myszkowski in his unpublished PhD thesis \cite{Mysz}, but the result follows also from the paper of Blanchard and Franks \cite{Bla-Fra}, and we present this approach in Theorem \ref{prop:lefschetz_numbers_of_odd_iterations_are_zero_for_reversing_orientation}. This result corrects the wrong statement regarding $S_3$ in  \cite{LlSi1} (Theorem 5.7).
Finally, we show that for any finite set  $\cA \subset \NN$, there exists a non-orientable surface of genus $\g$ and a Morse--Smale diffeomorphism $f \colon N_\g\to N_\g$ such that $\cAP(f)= \cA$ (Theorem \ref{thm:main1}).

\begin{theorem}\label{thm:main1}
	Let $\cA$ be a finite set of natural numbers.
	There exist a Morse--Smale diffeomorphism $f$ on a closed surface such that $\cA = \cAP(f)$. The surface can be chosen both orientable and non-orientable. Moreover, $f$ can be orientation-reversing (in the orientable case) if and only if $\cA \subset 2\NN$.
\end{theorem}

One of the consequences of the construction provided in Theorem \ref{thm:main1}  is  an  exponential growth in $\g$ of the  number of different conjugacy classes of algebraically finite type mapping classes, so, consequently, of conjugacy classes of homotopy classes of the Morse--Smale diffeomorphisms on a surface of a given genus (Theorem~\ref{thm:estimate of algebraically finite}).

Another important statement that we draw from our main theorem is the fact that for Morse--Smale diffeomorphisms (as well as more for a more general class of transversal maps) we are able to provide not only the realization of algebraic periods $n\in \cAP (f)$, but also an $n$-periodic point for $n$ odd, see Subsection
\ref{trans}.



\section{Description of the homotopy classes of homeomorphisms of surfaces which contain the Morse--Smale homeomorphism}\label{sec:homotopy classes of MS}

In this section, we give a brief survey of known results that describe the homotopy classes of homeomorphisms, thus diffeomorphisms,   of  surfaces which contain the Morse--Smale diffeomorphisms.  From now on, we assume that our surface is hyperbolic, i.e. it  is of genus $\g\geq 2$. The case of surfaces with non-negative Euler characteristic was studied separately in \cite{Rocha}.

First, we recall the Thurston classification theorem \cite{Farb-Marg,FLP}, also known as the  Nielsen--Thurston classification.

Let $f$ be a homeomorphism of  a closed orientable surface $S_\g$
of genus $ \g \geq   2$. Denote by $[f]$  the set of all homeomorphisms of the surface $S_\g$ that
are homotopic to $f$  (the homotopy class of homeomorphisms containing $f$).
According to the Nielsen--Thurston classification (see  \cite[Thms 11.6,  11.7]{FLP} and 
\cite{Gri-Mor-Poch}),
the set of all homotopy classes of homeomorphisms on $S_\g$  is represented as the
union of four disjoint subsets T1, T2, T3, and T4  distinguished by the conditions described below.

\begin{theorem}[Nielsen--Thurston]\label{thm:Thurston classification} Let $f$ be a homeomorphism of a closed orientable surface $S_\g$, $\g\geq 2$.
	\begin{itemize}
		\item[1.] {If  $[f] \in  T1$, then $[f]$ contains a periodic homeomorphism;}
		\item[2.] {if  $[f] \in T2$, then  $ [f]$  contains a reducible non-periodic homeomorphism of algebraically
			finite type;}
		\item[3.] {if $[f] \in   T3$, then $[ f]$  contains a reducible homeomorphism which is neither periodic nor  of algebraically finite type;}
		\item[4.] {if $[f] \in  T4$, then $[ f ]$ contains a pseudo-Anosov homeomorphism.}
	\end{itemize}
\end{theorem}

The classes $T1$ and $T2$ are called {\it algebraically finite type} in the original Nielsen terminology. Nielsen--Thurston theory was  developed for  orientable surfaces of genus $\g \geq  2$. The case $g=0$ (the sphere) is trivial, and for the case $g=1$ (the torus) the classification of homotopy classes of its homeomorphisms is provided algebraically by elements of $SL(2,\ZZ)$, which is a classical fact (cf. \cite{Farb-Marg}).

Roughly speaking, in the reducible case (T2 and T3),  Theorem \ref{thm:Thurston classification} states that one can cut a surface along the invariant collection of closed curves, obtaining a “canonical form”. This constitutes a collection of surfaces with boundary, satisfying that for each of them some iteration of the map is its self-homeomorphism. As a consequence,  any such mapping class can be reduced into either only periodic pieces (T2) or periodic and pseudo-Anosov pieces (T3). Thus, a mapping class is of algebraically finite type if all its pieces (even if one) in its Nielsen--Thurston reduction are periodic.

The analogue of Theorem \ref{thm:Thurston classification} for non-orientable surfaces was shown by   Yingqing Wu in \cite{Wu} (see \cite{Paris} for an exhaustive exposition). The statement of the mentioned version has the same formulation, and the main idea of its proof is to proceed with $\tilde{f}^+ \colon S_\g\to S_\g$ the preserving orientation lift of a homeomorphism $f \colon N_\g \to  N_\g$, where $p\colon S_\g \to N_\g$ is a two-sheets orientable cover of the non-orientable surface  $N_\g$.

Before our next consideration, let us recall the
Baer--Epstein theorem (cf. \cite{Farb-Marg}).\\

\begin{proposition}
	For homeomorphisms of a surface,  the homotopy is equivalent to the isotopy.
\end{proposition}  

Moreover, by the density argument, every homotopy class of homeomorphisms contains a diffeomorphism, and an isotopy between two diffeomorphisms can be replaced by a smooth isotopy. Accordingly,  in many formulations of Theorem \ref{thm:Thurston classification},  the isotopy classes are used instead of the homotopy classes.

The forthcoming part of the paper is based  on the following result of Luis~F. da Rocha \cite[THM A]{Rocha}

\begin{theorem}[L. F. da Rocha, 1985] \label{thm:da Rocha}
	If $M^2$ is a two-dimensional compact connected boundaryless manifold (orientable or not)
	with a negative Euler characteristic, then
	the condition that $f$ is of algebraically finite type
	is necessary and sufficient for an isotopy class of $f$ of $M^2$ to be Morse--Smale, i.e. it contains a Morse--Smale representant.

	In the case $\chi(M^2 ) \geq  0 $, we have the following situation:
	\begin{itemize}
		\item[(a)] {if $M^2$ is the sphere, the projective plane, or the Klein bottle, then every isotopy
			class is Morse--Smale;}
		\item[(b)] if $M^2$ is the two dimensional torus, an isotopy class is Morse--Smale if and
		only if
		$$ H_1(f): H_1({\mathbb T}^2; \CC) \to  H_1({\mathbb T}^2; \CC) $$
		has eigenvalues in the set $\{\pm 1, \pm \imath, 1/2 \pm (\sqrt{3}/2)\imath, -1/2  \pm (\sqrt{3}/2)\imath\}$.
	\end{itemize}
\end{theorem}

\begin{remark}
	
	Theorem \ref{thm:da Rocha}  was reproved by   A.N. Bezdenezhykh and   V.Z. Grines~\cite{Bez-Grin} for the class $T1$, and    by V. Grines,  A. Morozov,  and  O.  Pochinka \cite{Gri-Mor-Poch} for the class $T2$ of homeomorphisms of orientable surfaces. The authors of the referred articles used subtle analytical methods, and provided constructions of such Morse--Smale diffeomorphisms, solving more advanced questions.
\end{remark}

\begin{proposition}\label{fact:bounded_finite_unipotent}
	For a homeomorphism  of $ f \colon M \to M$  of a compact surface $M$, the following are equivalent:
	\begin{itemize}
		\item[1)]{$(L(f^n))$  is bounded,}
		\item[2)]{ $\cAP(f)$ is finite,}
		\item[3)]{$H_1(f)\colon H_1(M; \ZZ) \to H_1(M;\ZZ )$ is quasi-unipotent.}
	\end{itemize}
\end{proposition}
\begin{proof}
	The equivalence between the first and the second condition is given in \cite[Theorem 3.1.46]{Jez-Mar} (cf. \cite{Bab-Bog}). Since the Lefschetz numbers of iterations of  a diffeomorphism $f$ of a surface are equal to $1 - \tr H_1(f)^n + \pm 1 $ in the orientable case, and respectively  $1 - \tr H_1(f)^n$ in the non-orientable case, the first and third conditions are equivalent.
\end{proof}

However, let us mention here that the spectral radius of the  representation \mbox{$f \mapsto H_1(f)$} is not a strong enough invariant to determine whether $[f] \in T1 \cup T2$ (cf. \cite{FLP} or \cite{Farb-Marg} for the Thurston example of a pseudo-Anosov homeomorphism  $f$ with  $H_1(f)= {\rm id}$).

\section{Proof of main theorem}

In the following considerations we will use the notion of periodic expansion of an arithmetic function $\psi: \NN\to  \CC$ introduced in  \cite{Mar-NP}, exposed widely in \cite{Jez-Mar}, and used in \cite{GLM}, \cite{GLN-P}. In this language, an arithmetic function $\psi$ is represented as a series of elementary periodic functions
\begin{equation}\label{formula:k-periodic expansion}
	\psi(n)=\sum_{k\, | \, n} a_k(\psi) k = \sum_{k=1}^\infty a_k(\psi) \, \reg_k(n) \,,
\end{equation}
where  the coefficients $a_k(\psi)\in \CC $ are given by the M{\"o}bius inversion formula 
applied to the sequence $(\psi(n))$ analogously as it is in (\ref{equa:dual sequences}) for the sequence $(L(f^n))$,
and
\begin{equation}\label{formula:definition_of_reg_k}
	\reg_k(n)= \sum_{l=0}^{k-1} \left(e^{2\pi i \frac{l}{k}}\right)^n = \begin{cases} 0 \; \ {\text{if}}\; \ k\nmid n \cr    k \; \  {\text{if}}\; \  k\mid n \end{cases}
\end{equation}
is the sum of $n$-powers of all roots of unity of degree $k$ and is called the elementary periodic function. Conceptually, the periodic expansion is a discrete correspondent  of the Fourier expansion, arithmetically it is linked  with the Ramanujan sums, and its coefficients are   used to study the behavior of the number sequences which are related to the periodic points (see \cite{EPPW}). In our context, it is just  a convenient language   to handle the sequence we are studying.

Denote by $S_\g$ an orientable, and by $N_\g$ a non-orientable closed surface of genus $\g$. Similarly, $S_{\g,k}$ and $N_{\g,k}$ are compact surfaces of genus $\g$ with $k$ boundary components.

\begin{theorem}\label{thm:realization by algebraically finite}
	
	Let $\cA\subset \NN$ be finite.\begin{itemize}
		\item[(1)] There exist a closed orientable surface $S_\g$ and an orientation-preserving homeomorphism $f\colon S_\g \to S_\g$ of algebraically finite type such that $\cAP(f) = \cA$ 
		and the genus  $\g$ of $S_\g$ is equal to 
		$$
		\sum_{n \in \cA \setminus\{1\}} n\;\;\;{\text{ if}}\;\;\; 1\in \cA\;\;\; {\text{ and}}\;\;\;1+ \sum_{n\in \cA} n \;\;\;{\text{ if}}\;\;\; 1 \notin \cA\,.
		$$
		\item[(2)] There exist a closed orientable surface $S_\g$ and an orientation-reversing homeomorphism $f\colon S_\g \to S_\g$ of algebraically finite type such that $\cAP(f) = \cA$ if and only if $\cA \subset 2\NN$. Moreover, the genus $\g$ of $S_\g$ is equal to 
		$$
		\sum_{4 \,|\, n \in \cA}2n + \sum_{4\, \nmid\, n \in \cA\setminus\{2\}} n\;\;\;{\text{ if}}\;\;\; 2 \in \cA \;\;\;{\text{and}}\;\; \; 2+\sum_{4 \,|\, n \in \cA}2n + \sum_{4\, \nmid\, n \in \cA} n \;\;\;{\text{if}}\;\;\; 2 \notin \cA.
		$$
	\item[(3)] There exist a closed non-orientable surface $N_\g$ and a homeomorphism $f\colon N_\g \to N_\g$ of algebraically finite type such that $\cAP(f) = \cA$ and the genus $\g$ of $N_\g$ is equal to 
	$$
	\sum_{n\in \cA\setminus \{1\}}n \;\;\;{\text{if}} \;\;\; 1 \in \cA \neq \{1\}, \; 2+\sum_{n\in \cA} n \;\;\;{\text{if}}\;\;\; 1 \notin \cA \;\;\;{\text{and}}\;\;\; 1 \;\;{\text{if}}\;\; \cA = \{1\}.
	$$
\end{itemize}
\end{theorem}

\begin{proof}
First, let us describe the following construction. Having a finite set $\cA' \subset \NN$ and an assignment $\tau \colon \cA' \to \NN$, for each $n \in \cA'$ consider a surface $\Sigma_n$ which is an orientable surface $S_{\tau(n),2}$ in the orientable case or a non-orientable surface $N_{\tau(n),2}$ in the non-orientable case. On each $\Sigma_n$, we take a periodic homeomorphism $f_n$ of order~$\tau(n)$ which cyclically permutes its $1$-handles and is a rotation on each boundary component. 
More specifically, for an illustration, start with a cylinder $S^1 \times [-1,1]$ with a rotation by $1/\tau(n)$ of a full angle on each circle $S^1 \times \{t\}$, and remove $\tau(n)$ open disjoint discs $D_1,\ldots,D_{\tau(n)}$ in such a way that the rotation cyclically permutes the discs. Now, glue $\tau(n)$ copies of $S_{1,1}$ (or $N_{1,1}$) along permuted boundary components and extend the rotation to a periodic homeomorphism of order $\tau(n)$ that cyclically permutes the attached handles.

Next, having surfaces $\Sigma_{n_1},\ldots,\Sigma_{n_k}$, let us join $\Sigma_{n_i}$ and $ \Sigma_{n_{i+1}}$ for $1 \leq i \leq k-1$  by identifying their boundary components with end circles of a cylinder $C_i = S^1 \times [-1,1]$. By the construction, we get a surface $S$ with two boundary components and define $f$ to be equal to $f_{n_i}$ on $\Sigma_{n_i}$, and on each joining cylinder it is given by a homotopy between two rotations on end circles (see Figure~\ref{figure:periodic_homeo_preserv}).

\begin{figure}[htp]
	\centering

	\begin{tikzpicture}[scale=1]
		\node[rotate=0, scale=1] at (0,0) {\includegraphics[width=65mm]{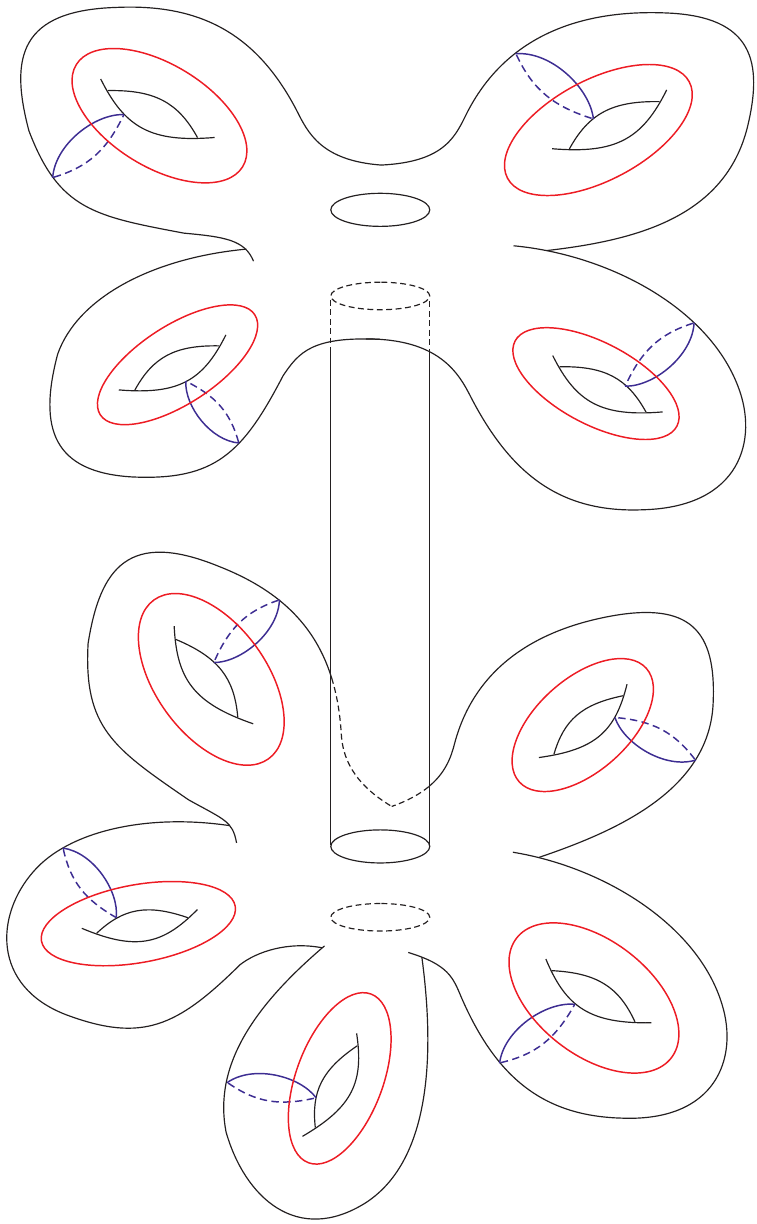}};

		\draw[->] (-2.8,3.4) to[out=205,in=155] (-2.8,2.4);
		\draw (-3.4,2.9) node {\Large $f_4$};

		\draw[->] (-2.8,-0.5) to[out=205,in=155] (-2.9,-1.7);
		\draw (-3.5,-1) node {\Large $f_5$};

		\draw (-0.5,4.9) node {\Large $\Sigma_4$};
		\draw (0.8,-4.7) node {\Large $\Sigma_5$};
		\draw (0.9,0.5) node {\Large $C_i$};

		\draw (-0.85,3.8) node {\large $a^4_1$};
		\draw (-0.85,2.8) node {\large $a^4_2$};
		\draw (1,2.7) node {\large $a^4_3$};
		\draw (0.82,3.6) node {\large $a^4_4$};
		
		\draw (-3.2,3.8) node {\large $b^4_1$};
		\draw (-0.9,1.2) node {\large $b^4_2$};
		\draw (3,2.7) node {\large $b^4_3$};
		\draw (0.88,5.1) node {\large $b^4_4$};

	\end{tikzpicture}

	\caption{Orientation-preserving case.
	}\label{figure:periodic_homeo_preserv}
\end{figure}

(1) In the case of an orientable surface and an orientation-preserving homeomorphism, take $\cA' = \cA \cup \{1\}$ if $1 \notin \cA$ or $\cA' = \cA \setminus \{1\}$ if $1 \in \cA$, with an assignment $\tau(n) = n$ for $n\in \cA'$. The above procedure provides a surface $S$ with an orientation-preserving homeomorphism~$f$ (see Figure \ref{figure:periodic_homeo_preserv}). Let us attach two discs to $S$ to get a closed orientable surface $\widetilde{S}$, and $f$ extends on $\widetilde{S}$, being an appropriate rotation on each of two discs. Taking one circle $S^1 \times \{0\}$ for each cylinder $C_i$, we get a system of simple closed curves preserved by $f$ such that $f$ is periodic on the complement of its open tubular neighborhood $\bigcup {\rm Int} C_i$, so $f$ is of algebraically finite type. Moreover, let $B_n = \{a^n_1,\ldots,a^n_{n},b^n_1,\ldots,b^n_{n}\}$ consist of circles in $S_{n}$ such that $\bigcup_{n\in \cA'} B_n$ forms the standard (symplectic) basis for $\widetilde{S}$. Then, the matrix of $H_1(f)$ on $H_1(\widetilde{S})$ is a block diagonal matrix whose blocks correspond to $B_n$'s, and since $f(a^n_j) = f_n(a^n_j) = a^n_{(j+1) \mod n}$ and similarly $f(b^n_j) = b^n_{(j+1)\mod n}$, the $n$th block is the direct sum of two permutation matrices for a cycle of length $n$:
$$
\begin{bmatrix}
	0 & 0 &\cdots & 0 & 1 \cr
	1 & 0 & \cdots & 0 & 0 \cr
	0 & 1 & \cdots & 0 & 0 \cr
	\vdots & \vdots & \ddots & \vdots & \vdots \cr
	0 & 0 & \cdots & 1 & 0
\end{bmatrix}.
$$
Therefore, the characteristic polynomial of $H_1(f)$ is equal to
$$
\prod_{n \in \cA'} (x^{n} - 1 )^2,
$$
and so it is not difficult to observe by definition of $\reg_k$ that $\tr(H_1(f^l)) = \sum_{n \in \cA'} \linebreak 2\reg_{n}(l)$. By the definition of $\cA'$,
\begin{equation}\label{L=regi}	
	L(f^l) = 2 - \tr(H_1(f^l)) = \sum_{n \in \cA} a_n(f) \, \reg_n(l),
\end{equation}
where $a_n(f) = -2$ for $n \in \cA \setminus \{1\}$, $a_1(f) = 2$ if $1\in \cA$, and $a_n(f) =0$ for $n\notin \cA$. Thus, $\cAP(f) = \cA$. 

(2) For the orientation-reversing case, we will show later that there are no odd algebraic periods
(see Theorem \ref{prop:lefschetz_numbers_of_odd_iterations_are_zero_for_reversing_orientation}). Assume that $\cA \subset 2\NN$ and take $\cA' = \cA \cup \{2\}$ if $2 \notin \cA$ or $\cA' = \cA \setminus \{2\}$ if $2 \in \cA$ with an assignment $\tau(n) = n$ if $4$ divides $n$ and $\tau(n) = n/2$ otherwise. Again, the procedure provides an orientable surface $S$, and we attach a disc to one of its two boundary components (also naming the resulting surface  $S$). We take its double $D(S)$, i.e. consider a copy $S'$ of $S$ consisting of surfaces $\Sigma'_{n}$ and cylinders $C'_i$, and join $S$ and $S'$ by a new cylinder $C = S^1 \times [-1,1]$ along their boundary components. We have defined $f$ on $S$ and $S'$, and on the new cylinder, let $f_{|S^1 \times \{t\}}$ be the same as $f_{|{S^1 \times \{\pm 1\}}}$.

Now, let $j \colon D(S) \to D(S)$ be  orientation-reversing involution mapping points of $S$ into corresponding points of $S'$, e.g., it can be seen as the reflection through the plane intersecting $D(S)$ in $S^1 \times \{0\} \subset C$ and separating symmetrically $S$ and $S'$, see Figure \ref{figure:periodic_homeo_revers}. The desired self-homeomorphism of $D(S)$ is $f' = f \circ j$. Indeed, cutting $D(S)$ along circles from connecting cylinders, we get the decomposition into surfaces $\Sigma_{n}$ and $\Sigma'_{n}$ such that $f'$ maps $\Sigma_{n}$ homeomorphically onto $\Sigma'_{n}$ and conversely. Thus, $(f')^2$ is a self homeomorphism of each piece $\Sigma_{n}$ or $\Sigma'_{n}$ on which one can check it is $n/2$-periodic. Hence, $f'$ is of algebraically finite type.

\begin{figure}[htp]
	\centering	
	
	\begin{tikzpicture}[scale=1.1]

		\node[rotate=0, scale=1] at (7.8,0) {\includegraphics[width=65mm]{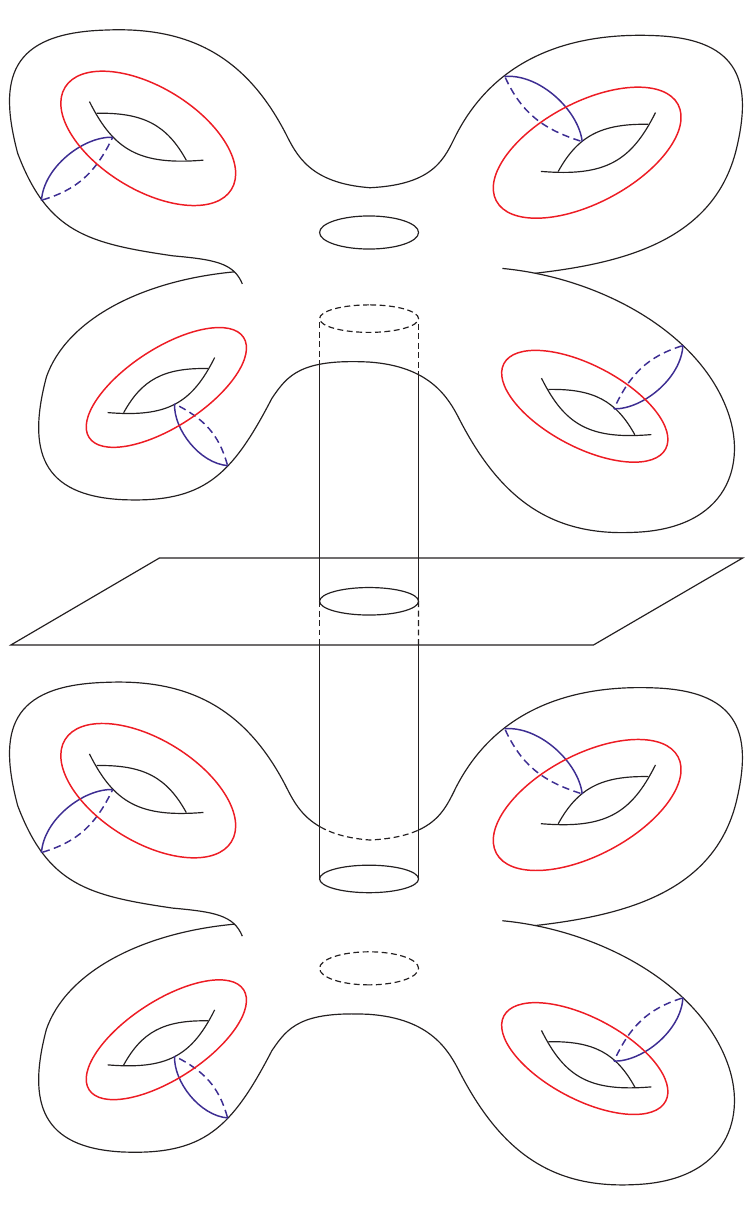}};

		\draw[->] (4.8,3.4) to[out=205,in=155] (4.8,2.4);
		\draw (4.2,2.9) node {\Large $f_4$};
		
		\draw[->] (4.8,-2.4) to[out=205,in=155] (4.8,-3.4);
		\draw (4.2,-2.9) node {\Large $f_4$};

		\draw (7.2,4.7) node {\Large $\Sigma_4$};
		
		\draw (8.4,-4.7) node {\Large $\Sigma'_4$};
		
		\draw[<->,thick] (8.8,-0.7) to (8.8,0.7);
		\draw (9.1,0.1) node {\Large $j$};

	\end{tikzpicture}

	\caption{Orientation-reversing case.
	}\label{figure:periodic_homeo_revers}
\end{figure}

Moreover, it is clear that $(f')^2(a^n_j) = a^n_{(j+2) \mod \tau(n)}$, $(f')^2(b^n_j) = b^n_{(j+2)\mod \tau(n)}$ on $\Sigma_{n}$ and similarly for $\Sigma'_{n}$. If $4$ divides $n$, then $\tau(n)=n$, and so $(f')^2_{|\Sigma_n}$ has order $n/2$, thus the $n$th block of the matrix of $H_1(f')$ is the direct sum of four permutation matrices for a cycle of length~$n$. However, if $n$ is not divisible by $4$, then $\tau(n) = n/2$ is odd and $(f')^2_{|\Sigma_n}$ still has order $\tau(n)=n/2$. Therefore, the characteristic polynomial of $f'$ is equal to
$$
\prod_{4 \,|\, n \in \cA'} (x^n-1)^4 \prod_{4\, \nmid\, n \in \cA'} (x^n -1)^2
$$
and so
$$
L((f')^l) =  2\reg_2(l) - \tr(H_1((f')^l)) = \sum_{n \in \cA} a_n(f')\, \reg_n(l),
$$
where $a_n(f') = -4$ for $n \in \cA$ such that $4 \,|\,n$, $a_n(f') = -2$ for $n \in \cA \setminus \{2\}$ such that $4 \,\nmid\,n$ and  $a_2(f') = 2$ if $2\in \cA$. Obviously, $a_n(f') = 0$ for $n \notin \cA$, so $\cAP(f') = \cA$.

(3) Analogously, we construct a desired self-homeomorphism of a non-orientable surface. If $\cA = \{1\}$, then just take ${\rm id}_{\RR\!P^2}$. Otherwise, consider $\cA' = \cA \cup \{1\}$ if $1 \notin \cA$ or $\cA' = \cA \setminus \{1\}$ if $1 \in \cA$ with an assignment $\tau(n)=n$ for $n\neq 1$ and $\tau(1) =2$ if $1\in \cA'$. Moreover, in the preliminary construction for the non-orientable case, if $1\in \cA'$, we take $f_1 = {\rm id}$ on $\Sigma_1$ which is of genus $\tau(1)=2$. We get a non-orientable surface $S$ with two boundary components to which we attach discs as previously. Similarly as in (1), we have a self-homeomorphism $f$ of $S$ of algebraically finite type, which is $n$-periodic on each $\Sigma_n$, except $n=1$.

Let $B_n = \{a^n_1,\ldots,a^n_{\tau(n)}\}$ be a set of circles in $\Sigma_n$ which are cyclically permuted by $f$ and such that elements of $\bigcup B_n$ represent standard generators of $H_1(S) \cong \ZZ^{g-1} \oplus \ZZ_{/2}$ with the relation $2\sum_{i,n}a^n_i = 0$, where $\g$ is the genus of $S$. 
The trace of $H_1(f)$ is computed as the trace of $\widetilde{H_1(f)}$ on $H_1(S)/{\rm T} \cong \ZZ^{\g-1}$, which is induced by $H_1(f)$ after dividing by the torsion part ${\rm T}$ of $H_1(S)$. Thus, $H_1(S)/{\rm T} \cong \left( \bigoplus_{i,n} \ZZ a^n_i \right)/ \left< \sum_{i,n} a^n_i \right>$ has a basis $\left(\bigcup B_n\right) \setminus \{a\}$ for some fixed $a\in B_{n_0}$, and $a$ is represented as the vector $(-1,\ldots,-1)$ in this basis. Hence, the matrix of $\widetilde{H_1(f)}$ is a block matrix if $n_0 =1 \in \cA'$, or almost a block matrix if $n_0 \neq 1$, i.e. the structure of a block matrix is disturbed only in the column which corresponds to a basis element $a^{n_0}_i$ such that $f(a^{n_0}_i) = a$  (if we order the basis in such a way that $a^{n_0}_i$ is the last element, then the matrix is upper block triangular). However, in the second case, the characteristic polynomial of $\widetilde{H_1(f)}$ is still the product of characteristic polynomials of permutation matrices for cycles of length $n$ for $n\neq n_0$, the identity matrix of dimension $2$ if $1\in \cA'$, and the characteristic polynomial of \smallskip
\begin{equation}\label{almost}
	\begin{bmatrix}
		0 & 0 &\cdots & 0 & -1 \cr
		1 & 0 & \cdots & 0 & -1 \cr
		0 & 1 & \cdots & 0 & -1 \cr
		\vdots & \vdots & \ddots & \vdots & \vdots \cr
		0 & 0 & \cdots & 1 & -1
	\end{bmatrix},
\end{equation}
which is the companion matrix of the polynomial $x^{n_0-1}+\ldots x + 1 = (x^{n_0} -1 ) / (x-1)$.
Since the trace of the companion matrix of a polynomial is equal to the sum of its roots, by  formula (\ref{formula:definition_of_reg_k}), the trace of the $l$-power of the matrix (\ref{almost}) is equal to $\reg_{n_0}(l) - \reg_1(l)$. Thus, in both the cases we get
$$
\tr({H_1(f^l)}) = \sum_{n \in \cA'\setminus \{1\}} \reg_{n}(l) + \sum_{1 \in\cA'}2\reg_1(l) - \reg_1(l),
$$
so
$$
L(f^l) = 1 - \tr({H_1(f^l)}) = \sum_{n \in \cA} a_n(f)\, \reg_n(l),
$$
where $a_n(f) = -1$ for $n \in \cA \setminus \{1\}$, $a_1(f) = 2$ if $1\in \cA$,  and $a_n(f) =0$ for $n\notin \cA$. Hence, $\cAP(f) = \cA$.
\end{proof}

To complete the proof, we are left with showing that  $\cAP(f) \cap (2\NN-1) = \emptyset$ if $f$ is orientation-reversing because we have shown only the realization for $\cAP(f) \subset 2\NN$.

In order to do this, we will need the notion of antisymplectic maps related to the maps induced on homology.

Each orientation-preserving homeomorphism $h\colon S_\g \to S_\g$ induces the linear map $H_1(h)$ that preserves the intersection form on {\small $H_1(S_\g;\mathbb{Z}) \cong \mathbb{Z}^{2\g} = \left<a_1,\ldots,a_\g,b_1,\ldots,b_\g\right>$},
$$
<\cdot,\cdot> \colon \mathbb{Z}^{2\g} \times \mathbb{Z}^{2\g} \to \mathbb{Z}
$$
given by
$$
<a_i,b_j > = \delta_{ij} = - <b_j,a_i>, \ \ \  <a_i,a_j> = 0, \ \ \ <b_i,b_j > = 0,
$$
whose matrix is
$$ \Omega = \begin{bmatrix}
0 & I_\g \cr
-I_\g & 0
\end{bmatrix}.
$$
Thus, if $A$ is the matrix of $H_1(h)$ in the basis $\left<a_1,\ldots,a_\g,b_1,\ldots,b_\g\right>$, then
\begin{equation}\label{symplectic_matrix}
\Omega = A^T \Omega A,
\end{equation}
so $A$ is a symplectic matrix.

Suppose $h\colon S_\g \to S_\g$ is an orientation-reversing homeomorphism. Thus, it induces $-\rm{id}$ on $H_2(S_\g) = \ZZ$, and so the matrix $A$ of $H_1(h)$ in the standard basis satisfies
\begin{equation}\label{anti}
A^T \Omega A = -\Omega.
\end{equation}
Such a matrix  is called {\it antisymplectic}.

In particular, one can show that $\det(A) = (-1)^\g$. The matrix $M = I_\g \oplus (-I_\g)$ is clearly antisymplectic, and $\det(M)=(-1)^\g$. If $A$ is antisymplectic, then $AM$ is symplectic, and so $\det(AM) = 1$. Therefore, $\det(A) = (\det(M))^{-1} = (-1)^\g$.

It is not difficult to verify that for a square matrix $A$ of dimension $n$, its characteristic polynomial $\chi_A(x)$ satisfies
$$
\chi_{A^{-1}}(x) = \frac{1}{\det(-A)}x^{n}\chi_A\left(\frac{1}{x}\right) \text{ \ \ and \ \ }\chi_{-A}(x) = (-1)^n \chi_A(-x).
$$

Therefore, if $A$ is antisymplectic of dimension $2\g$, then $A^T\Omega = -\Omega A^{-1}$, and so
\begin{equation}
\begin{aligned} \label{formula:characteristic_pol_antisymplectic}
	\chi_A(x) &= \det(xI -A) = \det(xI - A^T)\det(\Omega) = \det(x\Omega - A^T\Omega) \\
	&= \det(x\Omega + \Omega A^{-1}) = \det(xI - (-A^{-1})) = \chi_{-A^{-1}}(x)
	\\
	&= (-1)^{2\g}  \frac{1}{\det(-A)}(-x)^{2\g}\chi_A\left(-\frac{1}{x}\right) = (-1)^\g x^{2\g}\chi_A\left(-\frac{1}{x}\right).
\end{aligned}
\end{equation}

\begin{theorem}[cf. \cite{Bla-Fra}]
\label{prop:lefschetz_numbers_of_odd_iterations_are_zero_for_reversing_orientation}
Let $h\colon S_\g \to S_\g$ be an orientation-reversing quasi-unipotent homeomorphism (e.g. Morse--Smale diffeomorphism). Then, $L(h^{odd}) = 0$, and so $\cAP(h) \subset 2\NN$, i.e. $h$ has no odd algebraic periods.
\end{theorem}

\begin{proof}
For $n$ odd, $h^n$ is also orientation-reversing, so $L(h^n) = \tr(H_1(h^n))$ and $H_1(h^n)$ is antisymplectic. Therefore, it suffices to show that $L(h)=0$ just for $h$.

The characteristic polynomial $p(x)$ of $H_1(h)$ satisfies (\ref{formula:characteristic_pol_antisymplectic}), so if $p(x) = (x- \xi)^k q(x)$, then 
$$
p(x) = (x- \xi)^k q(x) = (-1)^g x^{2g}\left(-\frac{1}{x} - \xi\right)^k q\left(-\frac{1}{x}\right).
$$
Thus, if $\xi$ is a root of $p(x)$ of multiplicity $k$, then $p(x)$ has a root $-1/\xi$ of multiplicity $k$. In particular, if $\xi$ is a primitive root of unity of odd order $l$, then $-1/\xi$ is a primitive root of unity of order $2l$, and conversely, if $\xi$ is primitive of order $2l$, $l$ odd, then $\xi^l = -1$, so $-1/\xi$ is primitive of order $l$.

Let $\varphi_l(x)$ be the $l$th cyclotomic polynomial. Then, $p(x)$ has a factor $\varphi_l(x)$, $l$ odd, with multiplicity $k$ if and only if it has a factor $\varphi_{2l}(x)$ with multiplicity $k$. Since the sum of the roots of $\varphi_m(x)$ is equal to $\mu(m)$, $L(h)$ is a sum of the summands: $\mu(l)+\mu(2l)= \mu(l) - \mu(l) = 0$ for each odd $l$ such that $\varphi_l(x)$ is a factor of $p(x)$, and $\mu(4m) = 0$ for each factor $\varphi_{4m}(x)$ of $p(x)$. Thus, $L(h) = 0$.

By the M{\"o}bius inversion formula,
$$
a_{n}(h) = \sum_{d | n} \mu(n/d) L(h^d) = 0
$$
for odd $n$ since all $d | n$ are also odd.
\end{proof}

\begin{proof}[Proof of Theorem \ref{thm:main1}] The statement of Theorem \ref{thm:main1} follows from Theorems \ref{thm:da Rocha},  \ref{thm:realization by algebraically finite}, and \ref{prop:lefschetz_numbers_of_odd_iterations_are_zero_for_reversing_orientation}. \end{proof}

\begin{remark}
Note that all the statements related to the genus $\g$ in Theorem \ref{thm:realization by algebraically finite} remain valid for Morse--Smale diffeomorphisms in Theorem \ref{thm:main1}.
\end{remark}

\begin{remark}
\label{LLibre-bledy}
In fact, Llibre and coauthors considered the orientation-reversing case in \cite{LlSi1} (respectively, orientation-preserving case  in papers cited in \cite{LlSi3}), and they did not take into account the topological restrictions that come from the structure of the cohomological ring which are encoded in  condition (\ref{anti}) (or respectively in (\ref{symplectic_matrix})), which resulted in these papers in listing  also  non-realizable algebraic periods. In particular,  as we showed above, these restrictions  force that there are no odd numbers in the set of algebraic periods for orientation-reversing homeomorphisms.
\end{remark}

\subsection{Existence of periodic points for Morse--Smale diffeormophisms and transversal maps}\label{trans}

To formulate consequences of our main result related to the existence of periodic points, we recall the notion of a transversal map.

\begin{definition}[\cite{Dold}, \cite{Franks}]\label{defi:transversal}
Let $f\colon  \cU \to M$
be a~$C^1$-map of an open subset $\cU$ of a
manifold $M$. 
We say that
$f$ is transversal if for
any $m \in {\NN}$ and $x \in P^m(f)$,
$
1 \notin  \sigma(Df^m (x))$, where $\sigma$ denotes the spectrum of the derivative $D$ of $f^m(x)$.

\end{definition}

The set of all transversal maps $\cU \to M$ is denoted by $C_T(\cU,M)$ or $C_T(M)$ if $\cU=M$. The~main property of the class of transversal maps is given in the following theorem 
(see \cite{Dold} if $M= \RR^d$, and \cite{Seidel} for the general case, also \cite{Jez-Mar} for an exposition).

\begin{theorem}\label{thm:Seidel}
The set $C^\infty_T(\cU, M)$ is generic in $C^0({\cU}, M)$,
i.e. it  is the intersection $
C^\infty_T(\cU, M) =\bigcap_{n=1}^{\infty} G_n$ where
$G_n$ is open and  dense in $C^0({\cU}, M)$ .
In particular, every map $ f\colon \cU\to M$ is homotopic to a~transversal map $h\colon  \cU\to M$.
\end{theorem}

A geometric property of a~transversal map  is as follows.
\begin{proposition}\label{prop:periodic isolated for transv}
Let $M$ be a closed manifold.
For any $f\in C_T(M)$ and every $m\in\NN$, the set $P^m(f)$
consists of isolated points.
\end{proposition}

It follows easily from the definition that Morse--Smale diffeomorphisms belong to $C_T(M)$ (cf. \cite{Smale}, \cite{Franks}). We denote 
the set of Morse--Smale diffeomorphisms by $DC_{MS}(M)$.

Finally, all mappings in $C_T(M)$, thus in 
$DC_{MS}(M)$, have the following geometric property, which states that non-vanishing of $a_n(f)$, $n$-odd, implies that $P_n(f)\neq \emptyset$, i.e. $ n\in \per(f)$.

\begin{proposition}[see {\cite[Corollary 3.3.10]{Jez-Mar}} for an exposition, cf. \cite{Dold,Franks} for the case of Morse--Smale diffeomorphisms]\label{prop:periodic for transversal} {\phantom{newline}}

Let $f \colon M\to M$  be a transversal map.
\begin{center}
	If $a_n(f) \neq 0\,$ then
	$ \;\begin{cases} P_n(f) \cup  P_{\frac{n}{2}}(f) \neq  \emptyset \;  \, {\text{if}}\; n \; {\text{is even}}, \cr
		P_n(f) \neq  \emptyset \; \; {\text{if}}\; n\; {\text{is odd}}.
	\end{cases}
	$	
\end{center}
\end{proposition}

The above proposition allows us to formulate the analytical (dynamical) consequence of Theorem \ref{thm:main1}. Let $\cAP_{odd}(f) = \cAP(f)\cap (2\NN -1)$, and $\cAP_{even}(f)=\cAP(f)\cap (2\NN)$.

\begin{corollary}\label{coro: main analytical}
Let $\cA \subset \NN$ be a finite subset of natural numbers, and $f\colon S_\g\to S_\g$ be a preserving or reversing orientation Morse--Smale diffeomorphism of orientable surface of genus $\g$, or correspondingly a Morse--Smale diffeomorphism of a non-orientable surface $N_\g$  given by Theorem \ref{thm:main1} such that \mbox{$\cAP(f)=\cA$.}


Moreover, let $h \in C_T(S_\g)$, respectively $h \in C_T(N_\g)$, be homotopic to $f$ (in particular, $h$ can be equal to $f$). Then, $n\in \per(h)$ for every $n \in \cAP_{odd}(f)$ and $(\{n\} \cup  \{\frac{n}{2}\} )\cap \per(h) \neq \emptyset $ for every $n \in \cAP_{even}(f)$.
\end{corollary}

\section{Final discussion and applications}

In this section, we describe deep relations of the considered concepts with Nielsen periodic point theory. In the next step, we give an estimate from below of the number of algebraically finite type mapping classes of surface homeomorphisms.

First, we show that the fact that isotopy classes of Morse--Smale diffeomorphisms are contained in $T1 \cup T2$ follows also from the Nielsen number theory.

We recall that the asymptotic Nielsen number of a self-map $f\colon M\to M$ of a compact manifold $M$, correspondingly the asymptotic generalized  Lefschetz  number of $f$, are defined as
$$ N^\infty(f) =\limsup \sqrt[n]{N(f^n)}\,,\;\;\;{\text{and  respectively}}\;\;\;\;  L_\Gamma^\infty(f) =\limsup \sqrt[n]{\Vert L_\Gamma (f^n)}\Vert \,,$$
where the generalized Lefschetz number  $L_\Gamma(f)$ is defined as an element of the group ring $\ZZ(\Gamma)$ of $\Gamma= \pi_1(T_f)$ where $T_f$ is  the mapping torus $f$ (cf. (1.4) \cite{BoJia}).

For a surface without boundary, \cite[Thm 3.7]{BoJia} states the following: {\it If $f\colon M \to M$ is a homeomorphism of a compact connected surface with $\chi(M)<0$ (orientable or not orientable), then
$$N^\infty(f) = L_\Gamma^\infty(f) = \lambda,$$
where $\lambda$  is the largest stretching (expanding) factor of the pseudo-Anosov pieces in the
Thurston canonical form of $f$.}
Lemma \cite[Lem. 3.6]{BoJia} adds that $\log \lambda= {\bf h}(f)$ is equal to the topological entropy of the canonical representative of $f$.

\begin{remark}\label{rem:infinitely many periods}
In fact, a stronger result is established in \cite[Thm 3.7]{BoJia}. Namely, $N^\infty(f)=  NI^\infty(f)$, where the latter is equal to $\limsup \sqrt[n]{NP_n(f)}$, and $NP_n(f)$ is the so-called $n$-th Nielsen-Jiang prime periodic number (see \cite{BoJia} for a definition, or \cite{Jez-Mar} for a longer exposition). This invariant has the property $ NP_n(f) \leq \vert P_n(f)\vert$. It shows that in the case of a  homeomorphism  of  a surface $S_\g$, the condition  $N^\infty(f)=\lambda >1 $ implies that $P(f)={\underset{n=1}{\overset{\infty}\bigcup}} \,P^n(f)$, and also  ${\rm Per}(f)$,  are infinite. Since $N(f^n)$ and $NP_n(f)$ are homotopy invariants, this yields that a Morse--Smale diffeomorphisms can only occur in the classes $T1$ or $T2$. In other words, it establishes an implication in one direction of Theorem \ref{thm:da Rocha}:  The set  of the isotopy classes of Morse--Smale diffeomorphisms is contained in $T1 \cup T2$.
\end{remark}

\subsection{Mapping classes}

It is very useful to use the language of mapping classes in the study dynamics of homeomorphisms  of surfaces.  It is caused by  two factors. First,
we are interested in  homotopy  properties of homeomorphisms of surfaces. Second, the theory of mapping classes  groups is well-developed, and several theorems we use are formulated in these terms.

By the definition, {\it  the mapping class group} $\Mod(S_\g)$ of $S_\g$ is the quotient group
$$
{\rm Homeo}^+ (S_\g)/{\rm Homeo}_0(S_\g),
$$
where ${\rm Homeo}^+ (S_\g)$ denotes the group of preserving orientation homeomorphisms of $S_\g$, and ${\rm Homeo}_0(S_\g)$ its subgroup of homeomorphisms isotopic to the identity (cf. \cite{Farb-Marg}, \cite{Paris}).
Thus,  the assignment $ \phi \mapsto [\phi] \in \Mod(S_\g)$  is  a surjection from ${\rm Homeo}^+ (S_\g)$ onto $\Mod(S_\g)$.

This allows us to define an action (a representation) of the group $\Mod(S_\g)$
in $H_1(S_\g; \ZZ) \subset  H_1(S_\g; \RR)$. More precisely, for a given class $[\phi]$ of $\phi \in {\rm Homeo}^+ (S_\g)$, we define $\Psi([\phi]):= H_1(\phi) \in Aut(H_1(S_\g;\ZZ))) $, where
$H_1(\phi)$ is the induced homomorphism of $H_1(S_\g; \ZZ)$. It is well defined, i.e. it does not depend on a choice of representative of the class, because
the subgroup ${\rm Homeo}_0(S_\g)$ acts trivially on  $ H_1(S_\g; \ZZ)$. In fact, $H_1(\phi)$ belongs to the group $Sp(2\g,\ZZ)$ of symplectic matrices of size $n$ according to  property (\ref{symplectic_matrix}).

Now we can formulate a classical fact (that we use  
below to study transversal maps) about the mapping class group (cf. \cite[Proposition 7.3]{Farb-Marg}).

\begin{theorem}[H. Burkhardt (1889)]\label{thm:surjection on Sp(2g)}
The homomorphism {\small $\Psi  \colon \Mod (S_\g) \to  Sp(2\g,\mathbb Z)$} is surjective.
\end{theorem}

Another classical fact was shown by J. Nielsen by a direct geometrical consideration and a calculation of the characteristic polynomial   (cf. \cite{Nielsen}). Nowadays there are many ways of proving this statement.

\begin{theorem}[J. Nielsen 1944]\label{Nielsen}
$$If \;\; [f] \in T1\cup T2, \;\;{\text{then}}\;\; H_1(f)\;\;{\text{is quasi-unipotent}}\,.$$
\end{theorem}
Combining the   facts stated above, we get the following.

\begin{proposition}\label{prop:N^(f)=1 equivalent MS}
Every $f \in [h]$ of a homeomorphism  $h\colon M\to M$  of orientable or non-orientable surface $M$ such that $N^\infty(h)=1$ is isotopic to a Morse--Smale diffeomorphism, thus it is quasi-unipotent and $\cAP(f)$ is finite. Conversely, for every homeomorphism $f$ of $M$ which is homotopic to a Morse--Smale diffeomorphism,  we have $N^\infty(f)=1$.

Consequently, a class $[f]$ contains a Morse--Smale diffeomorphism if and only if it contains a homeomorphism $f^\prime$ with the entropy ${\bf h}(f^\prime)=0$.
\end{proposition}
\begin{proof}

If $N^\infty(h)=1$, then by the Boju Jiang theorem there is not a pseudo-Anosov piece  in the canonical Nielsen--Thurston form of $h$.
Indeed, this theorem states that  if  $N^\infty(h)= \lambda$, then $\lambda $ is the largest stretching factor of the pseudo-Anosov pieces  in the canonical form of $h$. However, for a pseudo-Anosov diffeomorphism we have $\lambda >1$, which leads to a contradiction if it would be such a piece.

Consequently, $[h] \in T1 \cup T2$. Now, by  the Nielsen theorem (Theorem \ref{Nielsen}), it is  quasi-unipotent. By  da Rocha's Theorem  \ref{thm:da Rocha}, there is a Morse--Smale diffeomorphism $\tilde{h} \in [h]$.

{Next, from Fact \ref{fact:bounded_finite_unipotent}, it follows that $\cAP(h)=\cAP(\tilde{h})$ is finite.}

{Finally, if $N^\infty(h)>1$, then  $h$ and every $h^\prime \sim h$  has infinitely many periodic orbits, which gives a contradiction if $h^\prime $ is a Morse--Smale diffeomorphism.}

The last part regarding entropy follows from the already quoted \cite[Lem. 3.6]{BoJia}.
\end{proof}

Let us remind that the mapping class group $\Mod(S_\g)$
can be also defined as the quotient
${\rm Diffeo}^+(S_\g)/{\rm Diffeo}_0(S_\g)$,
(cf. \cite{Farb-Marg}).

\begin{proposition}\label{transv_real}
For every symplectic or antisymplectic matrix $A\in Gl(2\g,\ZZ)$, there exists a transversal map $f \in C_T(S_\g) $ for which the induced automorphism $H_1(f)$  is equal to  $A$.
\end{proposition}

\begin{proof}
By Theorem \ref{thm:surjection on Sp(2g)},  for every $A\in Sp(2 \g, \ZZ)$, there exists a diffeomorphism $h$ of $S_\g$ for which the induced automorphism $H_1(h)$  is equal to $A$.
Next, we replace $h$ by a transversal map $f\in C_T(S_\g)$ homotopic to $h$ using Theorem~\ref{thm:Seidel}. 

In the case of an antisymplectic matrix $A$, take any orientation-reversing homeomorphism $h\colon S_\g \to S_\g$. Then, the product $AH_1(h)$ is symplectic, and by the first part of the proof induced by a transversal map $f'$. Thus, $f = f' \circ h^{-1}$ induces $H_1(f)=A$, and again can be approximated by a transversal map.
\end{proof}

\subsection{An estimate of the number of conjugacy classes of algebraically finite type mapping classes}

The formulas of Theorem \ref{thm:realization by algebraically finite} lead to an estimate from below of the number of conjugacy classes of algebraically finite type, mapping classes for a fixed genus~$\g$.
Let us consider two homeomorphisms $f', f''$ of a surface $M$.
If they are conjugated, then they induce conjugated homomorphisms on each homology group, in particular $\Psi(f^\prime)= H_1(f^\prime)$ and $\Psi(f^{\prime\prime})= H_1(f^{\prime\prime})$ are conjugated matrices in $Sp(2\g,\mathbf{Z})$. Since conjugated matrices have the same traces, $f^\prime$ and $f^{\prime\prime}$ have the same periodic expansions of sequences of Lefschetz numbers of iterations.
As a result, the number of different periodic expansions estimates the number of conjugacy classes of given elements in $\Mod(M)$.

This observation and our construction allow us to find an estimate of the number of conjugacy classes of homotopy classes of diffeomorphisms in $T1 \cup T2$.

\begin{definition}\label{partition}
\rm The number of ways of writing the
integer $N$ as a sum of positive integers, where the order of addends
is not considered significant, is denoted by $P(N)$ and  is called {\it the number of unrestricted partitions}.
\end{definition}

A partition of $N$ can be represented by a sequence $(p_1,\ldots,p_N)$ such that $N = \sum_{k=0}^N p_k k$, so $p_k \geq 0$ is the number of integers $k$ in the partition of $N$.

More information about the functions $P(N)$  can be found in \cite{Wei1}. At this moment, let us only mention
the asymptotic behavior of it (Hardy-Ramanujan 1918):
\begin{equation}\label{equa:asymptotic}
\;\;\;\; P(N)\,\sim\, \frac{1}{4 N\sqrt{3}} \, e^{\pi \sqrt{2N/3}}.
\end{equation}

\begin{theorem}\label{thm:estimate of algebraically finite}
The number of conjugacy classes of algebraically finite type mapping classes of an orientable or non-orientable closed surface of genus $\g$ is estimated from below by $P(\g)$, the number of unrestricted partitions of $\g$. Consequently, there are at least $P(\g)$ conjugacy classes of ($S_\g$ and $N_\g$) mapping classes containing Morse--Smale diffeomorphisms. 

Finally, the asymptotic growth in $\g$  of this number is greater than or equal to $ \frac{1}{4 \g\sqrt{3}} \, e^{\pi \sqrt{2\g/3}}\, $.
\end{theorem}

\begin{proof}

Let us first consider the orientable case. We will provide a correspondence between partitions $(p_1,\ldots,p_\g)$ of $\g$ and homeomorphisms $f\colon S_\g \to S_\g$ of algebraically finite type given by
\begin{equation}\label{formula:correspondence_partitions_and_dold_coeff_orient}
	a_n(f) = \begin{cases}
		-2p_n \ \ \ \ \ \ \ \ \ \text{ if } \ n\neq 1, \\
		-2(p_1-1) \ \; \text{ if } \ n =1,
	\end{cases}
\end{equation}
where $a_n(f)$ denote the coefficients in the periodic expansion of $f$.

We  apply the construction from the proof of Theorem \ref{thm:realization by algebraically finite}. Recall that in the orientation-preserving case (1), we realized $\cA$ as $\cAP(f)$ of a map $f \colon S_\g \to S_\g$ such that $a_k(f) = -2$ for $k\in \cA \setminus \{1\}$, $a_1(f) = 2$ if $1 \in \cA$, and obviously $a_k(f) =0 $ for $k \notin \cA$. From the description of the genus, or just from the construction, it follows that $2\g = 2-\sum_k a_k(f) \cdot k$, and so one can easily check that $\g = \sum_k p_k k$ by  formula (\ref{formula:correspondence_partitions_and_dold_coeff_orient}). Recall that the surface $S_\g$ was formed from the pieces $\Sigma_k$ for $k \in \cA'$ and connecting cylinders.

Now, for a given partition $(p_1,\ldots,p_\g)$ of $\g$ and every $k$, we take $p_k$ copies of the surface $\Sigma_k$ of genus $k$ together with defined periodic homeomorphisms, and glue them appropriately as in Theorem \ref{thm:realization by algebraically finite}. From the construction, it is straightforward that $\g = \sum p_k k$ is the genus of the resulted surface, and
$
\tr H_1(f^n) = \sum_k 2p_k \reg_k(n).
$
Therefore, we obtained an algebraically finite type homeomorphism $f$ of $S_\g$ such that
$$
L(f^n) = 2 - \tr H_1(f^n) = \sum_k a_k(f) \reg_k(n)
$$
is consistent with  formula (\ref{formula:correspondence_partitions_and_dold_coeff_orient}).

Now, consider the non-orientble case. The analogous correspondence between partitions $(p_1,\ldots,p_\g)$ of $\g$ and homeomorphisms $f\colon N_\g \to N_\g$ of algebraically finite type is given by
\begin{equation}\label{formula:correspondence_partitions_and_dold_coeff_non-orient}
	a_n(f) = \begin{cases}
		-p_n \ \ \ \ \ \ \text{ if } \ n\neq 1, \\
		2-p_1 \ \; \ \text{ if } \ n =1.
	\end{cases}
\end{equation}

Recall that in the construction in the proof of Theorem \ref{thm:realization by algebraically finite} in the non-orientable case (3), we realized $\cA$ as $\cAP(f)$ of $f\colon N_\g \to N_\g$ such that $a_k(f) = -1$ for $k \in \cA \setminus \{1\}$. Similarly, for a given partition $(p_1,\ldots,p_\g)$ of $\g$, we repeat the construction with $p_k$ copies of the non-orientable surface $\Sigma_k$ of genus $k$ on which the obtained function is periodic. Therefore, the obtained surface is clearly of genus $\g =  \sum p_k k$, and by the computations during the proof of Theorem \ref{thm:realization by algebraically finite}, we get
$$
\tr H_1(f^n) =  - \reg_1(n)+\sum_k p_k \reg_k(n).
$$
Thus, $f$ is an algebraically finite type homeomorphism of $N_\g$ such that
$$
L(f^n) = 1 - \tr H_1(f^n) = \sum_k a_k(f) \reg_k(n),
$$
where the coefficients $a_k(f)$ of the periodic expansion of $f$ are determined by (\ref{formula:correspondence_partitions_and_dold_coeff_non-orient}).

In both  cases, by Theorem \ref{thm:da Rocha} of da Rocha, the mapping class of $f$ contains a Morse--Smale representative. The relations (\ref{formula:correspondence_partitions_and_dold_coeff_orient}) and (\ref{formula:correspondence_partitions_and_dold_coeff_non-orient}) show that diffeomorphisms corresponding in our construction to different unrestricted partitions of $\g$ have different periodic expansions, and consequently their homotopy classes are not conjugated. 
\end{proof}

\appendix
\section{Algebraic periods and minimal sets of Lefschetz periods}\label{sec:Appendix2}

We recall that the zeta function for the sequence of Lefschetz numbers $L = (l_n) := (L(f^n))$ of iterations is defined as
\begin{equation}\label{equa:zeta}
\;\zeta_f(z) = \zeta(L;z):= \exp\Big({\underset{n=1}{\overset{\infty}\sum}} \,\frac{l_n}{n} z^n\Big).
\end{equation}
In  \cite{LlSi1}, the authors defined {\it the minimal set of Lefschetz periods} of a diffeomorphism $f\colon M\to M$, denoted ${\rm MPer}_L(f)$, in the following way.

\begin{definition}\label{defi::minimal Lefschetz periods}
\begin{equation}\label{equa:minimal Lefschetz periods}
	\begin{matrix} {\rm MPer}_L(f) := \bigcap \{ r_1, \, \dots \,, r_\eta \}
	\end{matrix}
\end{equation}
where the intersection is taken over all  representations of $\;\zeta(L;z) \;$  as
$$ \zeta(L;z)= {\underset{i=1}{\overset{\eta}\prod}} \, (1+\Delta_i\, z^{r_i})^{m_i} =
{\underset{j=1}{\overset{\eta^\prime}\prod}} \, (1- z^{r_j^{\prime}})^{m^\prime_j} \,
{\underset{k=1}{\overset{\eta^{\prime\prime}}\prod}} \, (1+ z^{r_k^{\prime\prime}})^{m^{\prime\prime}_k}\,, 
$$
where $r_i,r_j^{\prime},r_k^{\prime\prime} \in \NN$, $m_i,\, m_j^\prime, m_k^{\prime\prime} \in \ZZ$, and $\Delta_i = \pm 1$, i.e. we take into account  rational representations of $\zeta(L;z)$  as products of powers of polynomials $(1+\Delta_i\, z^{r_i})$.
\end{definition}

Geometrically, it is known (see the generalized Franks formula being a consequence of (3.3.9) in \cite{Jez-Mar}) that $\zeta(L;z)$ has such a representation, with each term related to some periodic orbit in case of transversal maps.

However, the factors and their amount are not canonically determined by the sequence $(L(f^n))$ in this formula. Llibre and coauthors derived the set ${\rm MPer}_L(f)$ for several examples (see the references of \cite{LlSi3}), observing  that it does not contain even numbers by a simple algebraic argument included here in the proof of Proposition \ref{prop:final algebraic periods} (cf. \cite{LlSi2}).
Note also that, formally, we have infinitely many of such rational representations, potentially with several possible values $\eta$, $r_i$, and $m_i$, which geometrically reflects the fact that periodic orbits of high periods may appear, whose contributions to $\zeta(L,z)$ annihilate one another.

Recently, in \cite{GLM}, the authors showed that for  a transversal  map $f\in C_T(M)$  of a manifold $M$, we have $ {\rm MPer}_L(f) = \cAP(f) \cap (2\NN-1)=\cAP_{odd}(f)$ using the periodic expansion of the sequence $(L(f^n))$. Originally, this theorem is stated  for Morse--Smale diffeomorphisms, but the argument holds for the class of transversal maps. It is worth pointing out that in \cite{GLM} the fact that all summands in representation of $(L(f^n))$  come from the geometric representation of $\zeta(L,z)$ is used.
Below, we present another purely algebraic way of showing that  ${\rm MPer}_L(f) =\cAP_{odd}(f)$ for a larger class of maps.

The Lefschetz zeta function $\zeta(L;z)$ is a rational function over $\ZZ$ (cf. \cite[(3.1.27)]{Jez-Mar}). Moreover,  $\zeta(L;z)$ has the following multiplicative representation (cf. \cite[(3.1.22)]{Jez-Mar}):
\begin{equation}\label{equa:multiplicative presentation of zeta}
\zeta(L;z)\,=\, {\underset{n=1}{\overset{\infty}{\,\prod\,}}}\, (1-z^n)^{a_n(f)}\,.
\end{equation}
Note that in the case when the sequence $(L(f^n))$ is bounded, e.g. if $f$ is a Morse--Smale diffeomorphism, the set $\cAP(f)=\{n: a_n(f)\neq0\}$ is finite, and consequently  the product (\ref{equa:multiplicative presentation of zeta}) is finite.

Moreover, the formula  (\ref{equa:multiplicative presentation of zeta})  is a unique rational expression of $\zeta_f(z)$ in terms of polynomials $(1-z^n)$ since none of the polynomials $(1-z^n)$ can be expressed as a rational function of polynomials $(1-z^k)$ for $k\neq n$.
Indeed, let
$
(1-z^n) = \prod_{k \neq n} (1-z^k)^{c_k}
$
for $c_k \in \ZZ$, and finitely many $c_k \neq 0$, and let $k_0$ be the largest number such that $c_{k_0} \neq 0$.
\begin{itemize}
\item If $k_0 < n$, then the left-hand side has a zero at a primitive root of unity of degree $n$, but the right-hand side does not.
\item If $k_0 > n$, then the right-hand side has a zero or pole at a primitive root of unity of degree $k_0$, but the left-hand side does not.
\end{itemize}

\begin{proposition}\label{prop:final algebraic periods}
Let $f \colon X\to X$ be a map of a finite CW-complex $X$ such that the sequence $(L(f^n))$ of Lefschetz numbers of iterations is bounded. Then,
$$
{\rm MPer}_L(f) = \cAP_{odd}(f).
$$
\end{proposition}

\begin{proof}
We have just three steps:
\begin{enumerate}
	\item The formula (\ref{equa:multiplicative presentation of zeta}) gives ${\rm MPer}_L(f) \subset \cAP(f)$.
	\item Since $(1-z^{2n}) = (1-z^n)(1+z^n)$ and $(1+z^{2n}) = \frac{1-z^{4n}}{(1-z^n)(1+z^n)}$, 
	$
	{\rm MPer}_L(f) \cap 2\NN = \emptyset.
	$
	
	\item Finally, if
	\begin{equation}\label{equa:zeta_MPer_arbitrary_form}
		\zeta_f(z) = \prod (1-z^k)^{c_k} \cdot  \prod (1+z^k)^{d_k},
	\end{equation}
	then use $(1+z^k) = \frac{1-z^{2k}}{1-z^k}$ to write
	$$
	\zeta_f(z) = \prod (1-z^k)^{c_k} \cdot  \prod \frac{(1-z^{2k})^{d_k}}{(1-z^k)^{d_k}} = \prod (1-z^k)^{e_k},
	$$
	where $e_k = c_k + d_{k/2} - d_k$ if $k$ is even and $e_k = c_k - d_k$ if $k$ is odd.
	Since the form (\ref{equa:multiplicative presentation of zeta}) is unique, $e_k = a_k(f)$. This means that $c_n - d_n = a_n(f) \neq 0$ for $n \in \cAP_{odd}(f)$. In particular, $c_n \neq 0$ or $d_n \neq 0$, so $n\in {\rm MPer}_L(f)$ because (\ref{equa:zeta_MPer_arbitrary_form}) is arbitrary.
\end{enumerate}
Therefore, ${\rm MPer}_L(f) = \cAP_{odd}(f)$.
\end{proof}

\end{document}